\documentclass{amsart}
\usepackage[centertags]{amsmath}
\usepackage{amstext,amssymb,amsopn,amsthm}

\title
{The  fractional Hardy inequality with a remainder term}
\author[B{.} Dyda]
{Bart{\l}omiej Dyda}
\address{Institute of Mathematics and Computer Science, Wroc{\l}aw University of Technology,
Wybrze\.ze Wyspia\'nskiego 27,
50-370 Wroc{\l}aw, Poland}
\email{bdyda@pwr.wroc.pl}

\theoremstyle{plain}
\newtheorem{thm}{Theorem}

\newtheorem{lem}[thm]{Lemma}

\newtheorem{cor}[thm]{Corollary}

\theoremstyle{definition}

\theoremstyle{remark}
\newtheorem*{rem*}{Remark}

\newcommand{\R}{\mathbb{R}}

\DeclareMathOperator{\dist}{dist}
\DeclareMathOperator{\diam}{diam}
\DeclareMathOperator{\supp}{supp}

\begin{document}
\sloppy \footnotetext{
\\ Research partially supported by grant MNiSW 3971/B/H03/2009/37.}

\begin{abstract}
We calculate the regional fractional  Laplacian
on some power function on an interval.
As an application, we prove Hardy inequality with an extra term
for the fractional Laplacian on the interval with the optimal
constant. As a result,
we obtain the fractional Hardy inequality with best constant
and an extra lower-order term for general domains, following the method developed
 by M.~Loss and C.~Sloane \cite{LossSloane}.
\end{abstract}
\maketitle
\section{Main result and discussion}\label{sec:i}
The purpose of this note is to prove the Hardy-type
inequality with an extra term in
the interval $(-1,1)$ -- Corollary~\ref{hardy} -- by the method of \cite{Fitz}.
We will obtain it from the following formula
analogous to the 'ground state representation' \cite{MR2469027}.
\begin{thm}\label{hardyerrorterm}
Let $0<\alpha<2$.
Let $w(x)=(1-x^2)^{(\alpha-1)/2}$.
For every $u\in C_c(-1,1)$,
\begin{align}
\mathcal{E}(u)&:=
\frac{1}{2}
\int_{-1}^1 \! \int_{-1}^1
\frac{(u(x)-u(y))^2}{|x-y|^{1+\alpha}} \,dx\,dy\nonumber\\
&=
\frac{1}{2}
\int_{-1}^1 \! \int_{-1}^1
\left(\frac{u(x)}{w(x)} - \frac{u(y)}{w(y)}\right)^2 
   \frac{w(x)w(y)}{|x-y|^{1+\alpha}} \,dx\,dy\nonumber\\
&+
{2^\alpha \kappa_{1,\alpha}}
\int_{-1}^1 {u^2(x)}
\;
(1-x^2)^{-\alpha}dx
+ 
\frac{1}{\alpha}
\int_{-1}^1 {u^2(x)}
\;
\phi(x)(1-x^2)^{-\alpha}\,dx
\label{hardyeq};
\end{align}
where $\phi(x)=2^\alpha - (1+x)^\alpha - (1-x)^\alpha$
and
\begin{equation}
  \label{eq:ns}
\kappa_{n,\alpha}=
\pi^{\frac{n-1}{2}} \frac{\Gamma(\frac{1+\alpha}{2})}{\Gamma(\frac{n+\alpha}{2})}
\frac{B\left(\frac{1+\alpha}{2}, \frac{2-\alpha}{2}\right)
  -2^{\alpha}}{\alpha2^{\alpha}}\,.
\end{equation}
\end{thm}
Here $B$ is the Euler beta function, and 
$C_c(a,b)$ denotes the class of all the continuous functions
$u\colon(a,b)\to \R$ with compact support in $(a,b)$.

\begin{cor}\label{hardy}
Let $1<\alpha<2$.
For every $u\in C_c(a,b)$,
\begin{align}
\frac{1}{2}
\int_{a}^b \! \int_{a}^b
\frac{(u(x)-u(y))^2}{|x-y|^{1+\alpha}} \,dx\,dy\
  &\geq
{\kappa_{1,\alpha}}
\int_a^b {u^2(x)}
\;
\left(
 \frac{1}{x-a}+\frac{1}{b-x} \right)^\alpha dx\nonumber\\
&+ 
\frac{4 - 2^{3-\alpha}}{\alpha (b-a)}
\int_a^b {u^2(x)}
\;
\left(
 \frac{1}{x-a}+\frac{1}{b-x} \right)^{\alpha-1} dx. \label{Hardyin}
\end{align}
Constant $\kappa_{1,\alpha}$ in {\rm (\ref{Hardyin})}
may not be replaced by
a bigger constant.
\end{cor}
On the right-hand side of (\ref{Hardyin}) we note the function
$\left(
 \frac{1}{x-a}+\frac{1}{b-x} \right) \geq \dist(x,(a,b)^c)$, where
$\dist(x, (a,b)^c)$ equals the distance of $x$ to the complement
of $(a,b)$.
Analogous Hardy inequalities, 
involving the distance to the complement of a domain,
and with optimal constants
 were derived
in \cite{FLS, MR2469027,  KBBD-bc, FrankSeiringer, LossSloane}.

In fact, by the method developed in the paper \cite{LossSloane},
from Corollary~\ref{hardy}
one gets the fractional Hardy inequality with a remainder term for
general domains.

\begin{thm}\label{hardyomega}
Let $1<\alpha<2$ and
let $\Omega \subset \R^n$ be a bounded domain.
For any $u\in C_c(\Omega)$
\begin{align}
\frac{1}{2}
\int_{\Omega\times\Omega} &
\frac{(u(x)-u(y))^2}{|x-y|^{n+\alpha}} \,dx\,dy
  \geq
{\kappa_{n,\alpha}}
\int_\Omega \frac{u^2(x)}{M_\alpha(x)^\alpha}\,dx \nonumber\\
&+ 
\frac{\pi^{(n-1)/2} \Gamma(\frac{\alpha}{2}) (4 - 2^{3-\alpha})}
{\alpha\Gamma(\frac{n+\alpha-1}{2})}\;
\frac{1}{\diam \Omega}
\int_\Omega \frac{u^2(x)}{M_{\alpha-1}(x)^{\alpha-1}}\,dx, \label{hardyomegain}
\end{align}
where $M_{\beta}$ is defined in \cite[formula (6)]{LossSloane}.
In particular, if $\Omega$ is a bounded convex domain, then for any
$f\in C_c(\Omega)$ we have
\begin{align}
\frac{1}{2}
\int_{\Omega\times\Omega} &
\frac{(u(x)-u(y))^2}{|x-y|^{n+\alpha}} \,dx\,dy
  \geq
{\kappa_{n,\alpha}}
\int_\Omega \frac{u^2(x)}{\dist(x,\Omega^c)^\alpha}\,dx \nonumber\\
&+ 
\frac{\pi^{(n-1)/2} \Gamma(\frac{\alpha}{2}) (4 - 2^{3-\alpha})}
{\alpha\Gamma(\frac{n+\alpha-1}{2})}\;
\frac{1}{\diam \Omega}
\int_\Omega \frac{u^2(x)}{\dist(x,\Omega^c)^{\alpha-1}}\,dx, \label{hardyconvexin}
\end{align}
Constant $\kappa_{n,\alpha}$ in {\rm (\ref{hardyomegain})} 
and {\rm (\ref{hardyconvexin})}
may not be replaced by
a bigger constant.
\end{thm}

Finally, let us note that the  symmetric bilinear form obtained from
$\mathcal{E}$ by polarisation is, up to
some multiplicative constant, the Dirichlet form of the censored
stable process in $(-1,1)$, for the  definition
of its domain  see \cite{BBC}.
We may also write a killed process counterparts of
Theorem~\ref{hardyerrorterm} and Corollary~\ref{hardy},
and it is interesting that they have
then a  simpler form.
Namely, we get the following
\begin{cor}\label{hardyk}
Let $0<\alpha<2$ and
 $w(x)=(1-x^2)^{(\alpha-1)/2}$.
For every $u\in C_c(-1,1)$,
\begin{align*}
\frac{1}{2}
\int_\R \! \int_\R 
\frac{(u(x)-u(y))^2}{|x-y|^{1+\alpha}} \,dx\,dy
 = &
\frac{1}{2} 
\int_{-1}^1 \! \int_{-1}^1
\left(\frac{u(x)}{w(x)} - \frac{u(y)}{w(y)}\right)^2 
   \frac{w(x)w(y)}{|x-y|^{1+\alpha}} \,dx\,dy\\
&+
\frac{B(\frac{1+\alpha}{2}, \frac{2-\alpha}{2})}{\alpha}
\int_{-1}^1 {u^2(x)}
\;
(1-x^2)^{-\alpha}dx\\
\geq &
\frac{B(\frac{1+\alpha}{2}, \frac{2-\alpha}{2})}{\alpha 2^\alpha} 
\int_{-1}^1 {u^2(x)}
\;
\left( \frac{1}{x+1}+\frac{1}{1-x} \right)^\alpha
dx.
\end{align*}
The constant $\frac{B(\frac{1+\alpha}{2},
  \frac{2-\alpha}{2})}{\alpha 2^\alpha}$
in the above inequality may not be replaced by a bigger one,
and there is no remainder term of the form (\ref{Hardyin}).
\end{cor}



\section{Proofs}
We define
\[
 L u(x) =  \lim_{\varepsilon\to 0^+}
\int_{(-1,1)\cap \{|y-x|>\varepsilon\}}\frac{u(y)-u(x)}{|x-y|^{d+\alpha}} \,dy.
\]
Note that $L$ is, up to the multiplicative constant, the 
regional fractional Laplacian for an interval $(-1,1)$, see
\cite{MR2214908}.  

\begin{lem}\label{laplasjanup}
Let $p>-1$ and $u_p(x) = (1-x^2)^p$.
For $0<\alpha<2$ we have
\begin{align*}
 L u_p(x) &=
\frac{(1-x^2)^{p-\alpha}}{\alpha} \Bigg(
(1 - x)^\alpha + (1 + x)^\alpha
- (2p+2-\alpha)B(p+1,1-\alpha/2) \\
&
 +\, {p.v.}\int_{-1}^1 \frac{(1-wx)^{\alpha-1-2p}-1}{|w|^{1+\alpha}}\;
 (1-w^2)^p\,dw
\Bigg).
\end{align*}
If we denote the  principal value integral above by $I(p)$, then
\begin{align*}
I(\frac{\alpha-3}{2}) &= x^2 B(p+1,1-\alpha/2) \quad\textrm{if $1<\alpha<2$;} \\
I(\frac{\alpha-2}{2}) &= I(\frac{\alpha-1}{2}) = 0.
\end{align*}
\end{lem}
\begin{proof}
We have, by changing variable $w=y^2$ and then integrating by parts,
\begin{align*}
Lu_p(0) &=
2\lim_{\varepsilon\to 0^+}\int_\varepsilon^1 \frac{(1-y^2)^p-1}{y^{1+\alpha}}\,dy  \\
&=
 2\lim_{\varepsilon\to 0^+}
\left(
 \frac{1}{2} \int_{\varepsilon^2}^1 (1-w)^p w^{-1-\alpha/2} [(1-w)+w]\,dw
 - \int_{\varepsilon}^1 y^{-1-\alpha}\,dy
\right) \\
&= 2\lim_{\varepsilon\to 0^+}
\bigg(
 \frac{1}{\alpha} (1-\varepsilon^2)^{p+1} \varepsilon^{-\alpha}
- \frac{p+1}{\alpha} \int_{\varepsilon^2}^1 (1-w)^p w^{-\alpha/2}\,dw\\
&\qquad\qquad + \frac{1}{2} \int_{\varepsilon^2}^1 (1-w)^p w^{-\alpha/2} \,dw
+ \frac{1}{\alpha} - \frac{\varepsilon^{-\alpha}}{\alpha}
\bigg)
\end{align*}
It is easy to see that
\[
\lim_{\varepsilon\to 0^+}
\left(
 \frac{1}{\alpha} (1-\varepsilon^2)^{p+1} \varepsilon^{-\alpha}
 - \frac{\varepsilon^{-\alpha}}{\alpha}
\right) =
\lim_{\varepsilon\to 0^+}
 \frac{\varepsilon^{2-\alpha}}{\alpha}\;
\frac{(1-\varepsilon^2)^{p+1}-1}{\varepsilon^2}
 = 0.
\]
Hence
\[
Lu_p(0)=  \frac{2}{\alpha}(1-(p+1-\alpha/2)B(p+1,1-\alpha/2)).
\]

We have, for $x_0\in (-1,1)$,
\[
Lu_p(x_0) =
p.v. \int_{-1}^1 \frac{(1-y^2)^p-(1-x_0^2)^p}{|y-x_0|^{1+\alpha}}\,dy.
\]
We change the variable in the following way
\begin{align*}
w&=\varphi(y):=\frac{x_0-y}{1-x_0y}; \qquad y=\varphi(w);\\
\varphi'(y)&=\frac{x_0^2-1}{(1-x_0y)^2};\\
y-x_0&= \frac{w(1-x_0^2)}{wx_0-1};\\
1-y^2&= \frac{(1-x_0^2)(1-w^2)}{(wx_0-1)^2}.
\end{align*}
Moreover the principal value integral transforms well.
We obtain
\begin{align}
Lu_p(x_0) &=
(1-x_0^2)^{p-\alpha}
 p.v.\int_{-1}^1 \frac{(1-w^2)^p - (1-wx_0)^{2p}}{|w|^{1+\alpha}}\;
 (1-wx_0)^{\alpha-1-2p}\,dw\nonumber\\
&=
 (1-x_0^2)^{p-\alpha} \Bigg[ Lu_p(0) -
p.v.\int_{-1}^1 \frac{(1-wx_0)^{\alpha-1} - 1}{|w|^{1+\alpha}}\,dw\label{calkap}\\
&
 + p.v.\int_{-1}^1 \frac{(1-wx_0)^{\alpha-1-2p}-1}{|w|^{1+\alpha}}\;
 (1-w^2)^p\,dw \Bigg].\nonumber
\end{align}
We calculate the integral in (\ref{calkap})
\begin{align*}
I &:=
p.v.\int_{-1}^1 \frac{(1-wx_0)^{\alpha-1} - 1}{|w|^{1+\alpha}}\,dw
=
\lim_{\varepsilon\to 0^+} (J_\varepsilon(x_0) + J_\varepsilon(-x_0)),
\end{align*}
where
\begin{align*}
 J_\varepsilon(x_0) &= 
\int_\varepsilon^1 \frac{(1-wx_0)^{\alpha-1} -
   1}{w^{1+\alpha}}\,dw\\
 &=
\int_\varepsilon^1 \Big( \frac{1}{w} - x_0 \Big)^{\alpha-1}
\frac{dw}{w^2} - \frac{\varepsilon^{-\alpha}-1}{\alpha}\\
&=
\frac{1}{\alpha}\Big( \frac{1}{\varepsilon} - x_0 \Big)^\alpha
- \frac{1}{\alpha}\Big( 1 - x_0 \Big)^\alpha 
 - \frac{\varepsilon^{-\alpha}-1}{\alpha}\\
&=\frac{1}{\alpha} - \frac{1}{\alpha}\Big( 1 - x_0 \Big)^\alpha
+ \frac{(1-\varepsilon x_0)^\alpha - 1}{\alpha \varepsilon^{\alpha}}.
\end{align*}
By l'H\^ospital rule we find that
\begin{align*}
I &=
\frac{2}{\alpha} - \frac{1}{\alpha}\Big( 1 - x_0 \Big)^\alpha
- \frac{1}{\alpha}\Big( 1 + x_0 \Big)^\alpha,
\end{align*}
and the first part of the lemma is proved.

Obtaining $I(p)$ for the three values of $p$ is easy and is omitted.
\end{proof}

Next lemma is in fact a very special case of  Proposition~2.3 from
\cite{FrankSeiringer}. In this special case the proof
may be simplified, and for reader's convenience we give the sketch of
the proof.

\begin{lem}\label{lemfitz}
For every $u\in C_c(-1,1)$ and any strictly positive function $w\in
C^2(-1,1)$ such that $Lw\leq 0$, we have
\[
 \mathcal{E}(u) = \int_{-1}^1 u^2(x)
    \frac{-L w(x)}{w(x)}\,dx
+
\frac{1}{2}
\int_{-1}^1 \! \int_{-1}^1
\left(\frac{u(x)}{w(x)} - \frac{u(y)}{w(y)}\right)^2 
   \frac{w(x)w(y)}{|x-y|^{1+\alpha}} \,dx\,dy.
\]
\end{lem}

\begin{proof}
We use the elementary equality \cite[(9)]{KBBD-bc}
\begin{align*}
(u(x)-u(y))^2 &+ u^2(x) \frac{w(y)-w(x)}{w(x)}
+ u^2(y) \frac{w(x)-w(y)}{w(y)} \\
&= \left( \frac{u(x)}{w(x)} - \frac{u(y)}{w(y)} \right)^2
w(x)w(y),
\end{align*}
then integrate against the measure
$1_{\{|x-y|>\varepsilon\}}|x-y|^{-1-\alpha}\,dx\,dy$
and take $\varepsilon\to 0$. We use Taylor expansion for $w$
and the compactness of the support of $u$
 to apply Lebesgue dominated convergence theorem.
\end{proof}

\begin{proof}[Proof of Theorem~\ref{hardyerrorterm}]
 The theorem follows immediately from Lemma~\ref{laplasjanup}
applied to $w(x)=(x^2-1)^{(\alpha-1)/2}$ and Lemma~\ref{lemfitz}
applied to $p=(\alpha-1)/2$.
\end{proof}

\begin{proof}[Proof of Corollary~\ref{hardy}]
By Theorem~\ref{hardyerrorterm} and symmetry of the function $\phi$,
to prove (\ref{Hardyin})
it is enough to show that
\begin{equation}\label{claim}
 \phi(x) := 2^\alpha-(1+x)^\alpha - (1-x)^\alpha \geq (2^\alpha-2)(1-x^2), \quad x\in [0,1].
\end{equation}

After substitution $u=x^2$, we see that it is enough to show that
the function
\[
 g(u):= (2^\alpha-2)u -(1-\sqrt{u})^\alpha -(1+\sqrt{u})^\alpha + 2
\]
is concave. To show this, it suffices to show that
\[
 g'(u) = 2^\alpha-2 + \frac{\alpha}{2\sqrt{u}}
  \left(
  (1-\sqrt{u})^{\alpha-1} - (1+\sqrt{u})^{\alpha-1}
  \right)
\]
is a decreasing function. We again substitute $t=u^2$ and observe that
\[
 \frac{(1-t)^{\alpha-1} - (1+t)^{\alpha-1}}{t} = \frac{h(t) - h(0)}{t},
\]
where $h(t) = (1-t)^{\alpha-1} - (1+t)^{\alpha-1}$. Now since $h$ is
concave, the function $t\mapsto \frac{h(t) - h(0)}{t}$ is decreasing,
and so is the function $g'$. This proves the claim (\ref{claim}) and,
consequently, the inequality (\ref{Hardyin}).

 The fact that the constant $\kappa_{1,\alpha}$ in (\ref{Hardyin})
is optimal follows from \cite{KBBD-bc}.
\end{proof}

\begin{proof}[Proof of Theorem~\ref{hardyomega}]
The proof is analogous to the proof of Theorem~1.1 in
\cite{LossSloane}. In inequality \cite[(33)]{LossSloane}
we would get an extra lower order term. Namely, using notation
from \cite{LossSloane}, the extra term would be
\begin{align*}
\frac{1}{2}\lambda_{1,\alpha}\int_{S^{n-1}}dw& \int_{\{x:x\cdot w=0\}}
d\mathcal{L}_w(x)
\int_{x+sw\in \Omega} ds |f(x+sw)|^2 \\
\times&
  \left[ \frac{1}{d_w(x+sw)} + \frac{1}{\delta_w(x+sw)}
  \right]^{\alpha-1} 
 \frac{1}{d_w(x+sw) + \delta_w(x+sw)}\\
= 
\frac{1}{2}\lambda_{1,\alpha}\int_{S^{n-1}}dw&
  \int_\Omega |f(x)|^2
  \left[ \frac{1}{d_{w,\Omega}(x)} + \frac{1}{\delta_{w,\Omega}(x)}
  \right]^{\alpha-1}
 \frac{1}{d_{w,\Omega}(x) + \delta_{w,\Omega}(x)}\\
&\geq
\frac{\lambda_{n,\alpha}}{\diam\Omega}
\int_\Omega \frac{|f(x)|^2}{M_{\alpha-1}(x)^{\alpha-1}}\,dx.
\end{align*}
Thus we get (\ref{hardyomegain}).

Note that \cite[(9)]{LossSloane}
 is valid for any $\alpha>0$. We apply it to $\alpha-1$ in
place of $\alpha$ and we get inequality (\ref{hardyconvexin}).
\end{proof}

\begin{proof}[Proof of Corollary~\ref{hardyk}]
The equality follows from Theorem~\ref{hardyerrorterm} and
the following easy formula
\[
\frac{1}{2}
\int_\R \! \int_\R 
\frac{(u(x)-u(y))^2}{|x-y|^{1+\alpha}} \,dx\,dy
=
 \frac{1}{2} \mathcal{E}(u)
 + \int_{-1}^1 u^2(x) \frac{1}{\alpha} \left(
 (1+x)^{-\alpha} + (1-x)^{-\alpha} \right)\,dx.
\]
The sharpness of the constant and the lack of a remainder term
follow from approximating 
the function $w$, for example by $u=\psi_n w$, where
 $\supp \psi_n \subset [-1+1/n, 1-1/n]$,
$\psi_n(x)=1$ on $(-1+2/n, 1-2/n)$
and $|\psi_n'(x)| \leq 2n$ on $(-1,1)$.
\end{proof}

\bibliographystyle{abbrv}
\bibliography{bibca}

\end{document}